\documentclass[12pt, twoside, leqno]{article}
\usepackage{hyperref}

\usepackage{amsmath,amsthm}
\usepackage{amssymb}

\usepackage{enumitem}

\usepackage{graphicx}

\usepackage[T1]{fontenc}

\usepackage{amsfonts}
\usepackage{float}
\usepackage{cancel}

\newtheorem{theorem}{Theorem}[section]
\newtheorem{corollary}[theorem]{Corollary}

\theoremstyle{definition}

\numberwithin{equation}{section}

\begin{document}


\baselineskip=17pt


\title{Inequalities about the area bounded by three cevian lines of a triangle}

\author{Yagub N. Aliyev\\
School of IT and Engineering\\ 
ADA University\\
Ahmadbey Aghaoglu str. 61 \\
Baku 1008, Azerbaijan\\
E-mail: yaliyev@ada.edu.az}
\date{}

\maketitle


\renewcommand{\thefootnote}{}

\footnote{2020 \emph{Mathematics Subject Classification}: Primary 51M16, 51M25; Secondary 51M04, 52A38, 52A40, 97G30, 97G40.}

\footnote{\emph{Key words and phrases}: Steiner-Routh's theorem, Schlömilch's theorem, cevians, area, triangle, geometric inequality, Hölder's inequality, Rigby's inequality, Möbius theorem.}

\renewcommand{\thefootnote}{\arabic{footnote}}
\setcounter{footnote}{0}


\begin{abstract}
In the paper we prove generalization of Schlömilch's and Zetel's theorems about concurrent lines in a triangle. This generalization is obtained as a corollary of sharp geometric inequality about the ratio of triangular areas which is proved using discrete variant of Hölder's inequality. Also a new sharp refinement of J.F. Rigby's inequality, which itself generalized Möbius theorem about the areas of triangles formed by cevians of a triangle, is proved.
\end{abstract}

\section{Introduction}
Consider cevians $AD$, $BE$, and $CF$ of a triangle $ABC$ (see Fig. 1). Denote $\frac{|BD|}{|DC|}=\lambda_1$, $\frac{|CE|}{|EA|}=\lambda_2$, and $\frac{|AF|}{|FB|}=\lambda_3$. Denote also $BE\cap CF=G_1$, $AD\cap CF=G_2$, and $AD\cap BE=G_3$. There is a result in geometry known as Steiner-Routh's theorem which says that
$$\frac{\textnormal{Area}(\triangle G_1G_2G_3)}{\textnormal{Area}(\triangle ABC)}=\frac{(\lambda_1 \lambda_2 \lambda_3-1)^2}{(\lambda_1 \lambda_2+\lambda_1 +1)(\lambda_2 \lambda_3+\lambda_2 +1)(\lambda_3 \lambda_1+\lambda_3 +1)}. \eqno(1)$$
\begin{figure}[htbp]
\centerline{\includegraphics[scale=0.5]{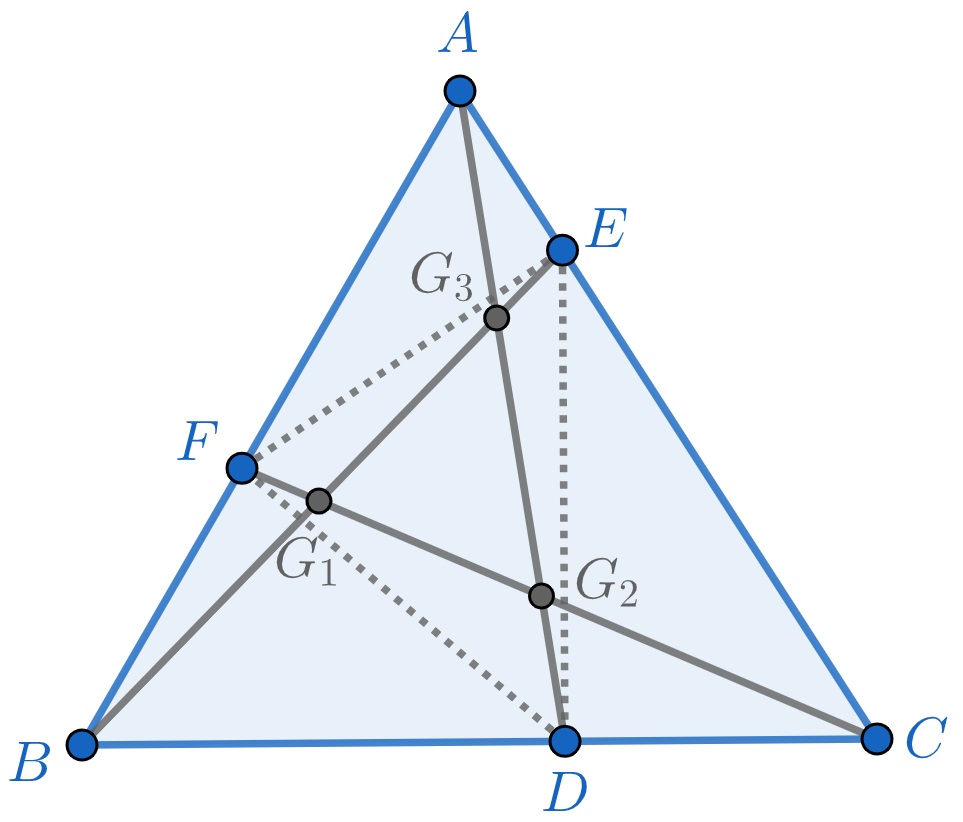}}
\label{fig1}
\caption{Steiner-Routh's theorem.}
\end{figure}
Steiner-Routh's theorem which is sometimes called just Routh's theorem was discussed in many papers and books. See \cite{steiner} p. 166, \cite{glais} p. 33, \cite{routh} p. 89, \cite{dor} p. 41-42, \cite{mik},
\cite{cox} p. 211, 212, \cite{abbud}, \cite{abbud2}, \cite{weis}, \cite{kay} p. 382,
 \cite{klam}, \cite{niven}, \cite{bot}, \cite{bol} p. 276, \cite{vol}, \cite{shir}, \cite{stan}, \cite{shmin}, \cite{deak}, and their references. Steiner-Routh's formula  was generalized in many different directions: \cite{benyi}, \cite{car}, \cite{marko}, \cite{litv}, \cite{zuo}. There is a peculiar special case called \textit{One seventh area triangle} or \textit{Feynman's triangle} which corresponds to case $\lambda_1 =\lambda_2=\lambda_3=2$ and attracted much attention because it can also be proved using dissections (see e.g. \cite{mik}, \cite{steinhaus} p. 9). Most of these sources also mention the following formula
$$\frac{\textnormal{Area}(\triangle DEF)}{\textnormal{Area}(\triangle ABC)}=\frac{\lambda_1 \lambda_2 \lambda_3+1}{(\lambda_1 +1)(\lambda_2 +1)(\lambda_3 +1)}. \eqno(2)$$
Formula (1) and (2) generalize Ceva's $(\lambda_1 \lambda_2 \lambda_3=1)$ and Menelaus' $(\lambda_1 \lambda_2 \lambda_3=-1)$ theorems, respectively. In these cases the areas of $\triangle G_1G_2G_3$ and $\triangle DEF$ are equal to zero, which is equivalent to say that cevians $AD$, $BE$, and $CF$ are concurrent, and points $D$, $E$, and $F$ are collinear, respectively. In general, the vertices of a triangle do not necessarily coincide if its area is zero. It is possible that the vertices of the triangle are just collinear. But this is not possible for $\triangle G_1G_2G_3$, because otherwise points $A,B,$ and $C$ would also be collinear. In the paper we will apply this idea to find a new proof for the following theorem and its generalization. 

\textbf{Schlömilch's theorem.} \textit{The lines connecting the midpoints of the sides of a triangle and the midpoints of the corresponding altitudes are concurrent.}

O. Schlömilch's theorem was discussed in many papers and books. See for example \cite{zvon}, \cite{oprea}, \cite{zetel}, , \cite{ponar} p. 34, 37, \cite{efrem} p. 133, 
\cite{alt} p. 256, 304. In \cite{john}, p. 215 (Corollary), \cite{prasolov} Problem 5.135 it was mentioned that the point of concurrency in Schlömilch's theorem is Lemoine (symmedian) point of the triangle. In \cite{zetel2} S.I. Zetel generalized the result by Schlömilch as follows (see Fig. 2).

\textbf{Zetel's theorem.} \textit{Let trio of cevians $AD$, $BE$, and $CF$ of a triangle $ABC$ be concurrent at point $G$. Let another trio of cevians $AK$, $BL$, and $CM$ of triangle $ABC$ be concurrent at point $H$. Denote $AK\cap EF=N$, $BL\cap DF=Q$, and $CM\cap DE=P$. Then lines $DN$, $EQ$, and $FP$ are concurrent.}

\begin{figure}[htbp]
\centerline{\includegraphics[scale=0.5]{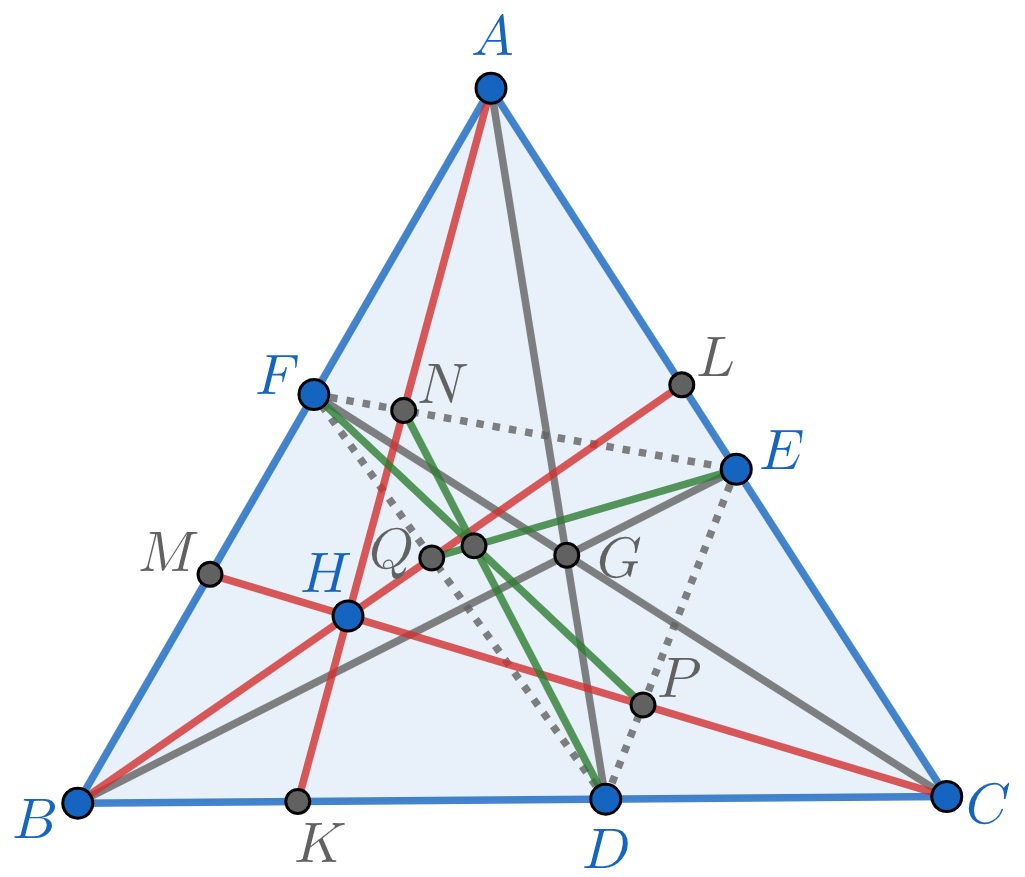}}
\label{fig2}
\caption{Zetel's generalization of Schlömilch's theorem.}
\end{figure}

From the viewpoint of projective geometry this generalization is equivalent to Schlömilch's theorem. Indeed, by Desarques' theorem the intersection points $EF\cap BC$, $DF\cap AC$, and $DE\cap AB$ are on a line. Let us apply a projective transformation sending this line to infinity. We will continue to use the original notations for their
images under these transformations. This transformation forces $EF||BC$, $DF||AC$, $DE||AB$, and therefore points $D,E,F$ are the midpoints of sides $BC$, $AC$, $AB$, respectively. Then apply affine transformations changing $AK$ and $BL$ to the corresponding altitudes of $\triangle ABC$. Then $CM$ is also the altitude of $\triangle ABC$, and therefore we return to Schlömilch's theorem. In the current paper we will obtain Zetel's generalization of Schlömilch's theorem and other similar theorems as corollaries of inequalities about triangular areas in the corresponding configurations. We will prove some of these inequalities using discrete version of Hölder's inequality \cite{bel} p. 20.

\textbf{Hölder's inequality (discrete case).} \textit{If $x_{ij}$ $(i=1,\ldots,n; j=1,\ldots,m)$ are non-negative numbers, $p_j>1$ and $\sum_{j=1}^m \frac{1}{p_j}=1$ then
$$
\sum_{i=1}^n\prod_{j=1}^m x_{ij}\le \prod_{j=1}^m\left(\sum_{i=1}^n x_{ij}^{p_j}\right)^{\frac{1}{p_j}}. \eqno(3)
$$}
Hölder's inequality made it possible to prove the inequalities in the current paper without any use of calculus.

We also considered the following result by J.F. Rigby \cite{rigby} (see also \cite{mit} p. 340).

\textbf{Rigby's inequality.} \textit{Let $p$, $q$, $r$, $x$, and $y$ denote the areas of $\triangle AEF$, $\triangle BFD$, $\triangle CDE$, $\triangle DEF$, and $\triangle G_1G_2G_3$ (Figure 1).
Then
$$
x^3+(p+q+r)x^2-4pqr\ge 0, \eqno(4)
$$
with equality if and only if cevians $AD$, $BE$, and $CF$ are concurrent.}

The equality case is 
known as Möbius' theorem \cite{mob} p. 198 (see also \cite{balk} p. 95).

\textbf{Möbius' theorem.} \textit{If cevians $AD$, $BE$, and $CF$ are concurrent then
$$
x^3+(p+q+r)x^2-4pqr=0.
$$}
In the current paper we will prove refinement of inequality (4):
$$
x^3+(p+q+r)x^2-4pqr\ge x^2y. \eqno(5)
$$
Interesting inequalities involving the areas in the above configurations also appeared in \cite{oxman}. 
\section{Main results}
First, general sharp inequality about the areas of triangles formed by cevians of a triangle, will be proved. After the proof its special cases corresponding to concurrent cevians will be discussed.
\begin{theorem} Let $D$ and $K$, $E$ and $L$, $F$ and $M$ be arbitrary points on sides $BC$, $AC$, and $AB$, repsectively, of a triangle $ABC$. Denote $AK\cap EF=N$, $BL\cap DF=Q$, $CM\cap DE=P$, $DN\cap EQ=R$, $FP\cap EQ=S$, and $FP\cap DN=T$. Denote also $\frac{|BD|}{|DC|}=\lambda_1$, $\frac{|CE|}{|EA|}=\lambda_2$, $\frac{|AF|}{|FB|}=\lambda_3$, $\frac{|BK|}{|KC|}=u$, $\frac{|CL|}{|LA|}=v$, and $\frac{|AM|}{|MB|}=w$.
Then
$$
\frac{\textnormal{Area}(\triangle RST)}{\textnormal{Area}(\triangle DEF)}\le\frac{(\lambda_1 \lambda_2 \lambda_3uvw-1)^2}{\left(\sqrt[3]{(\lambda_1 \lambda_2\lambda_3uvw)^2}+\sqrt[3]{\lambda_1 \lambda_2\lambda_3uvw}+1\right)^3}. \eqno(6)
$$
\end{theorem}
\begin{figure}[htbp]
\centerline{\includegraphics[scale=0.4]{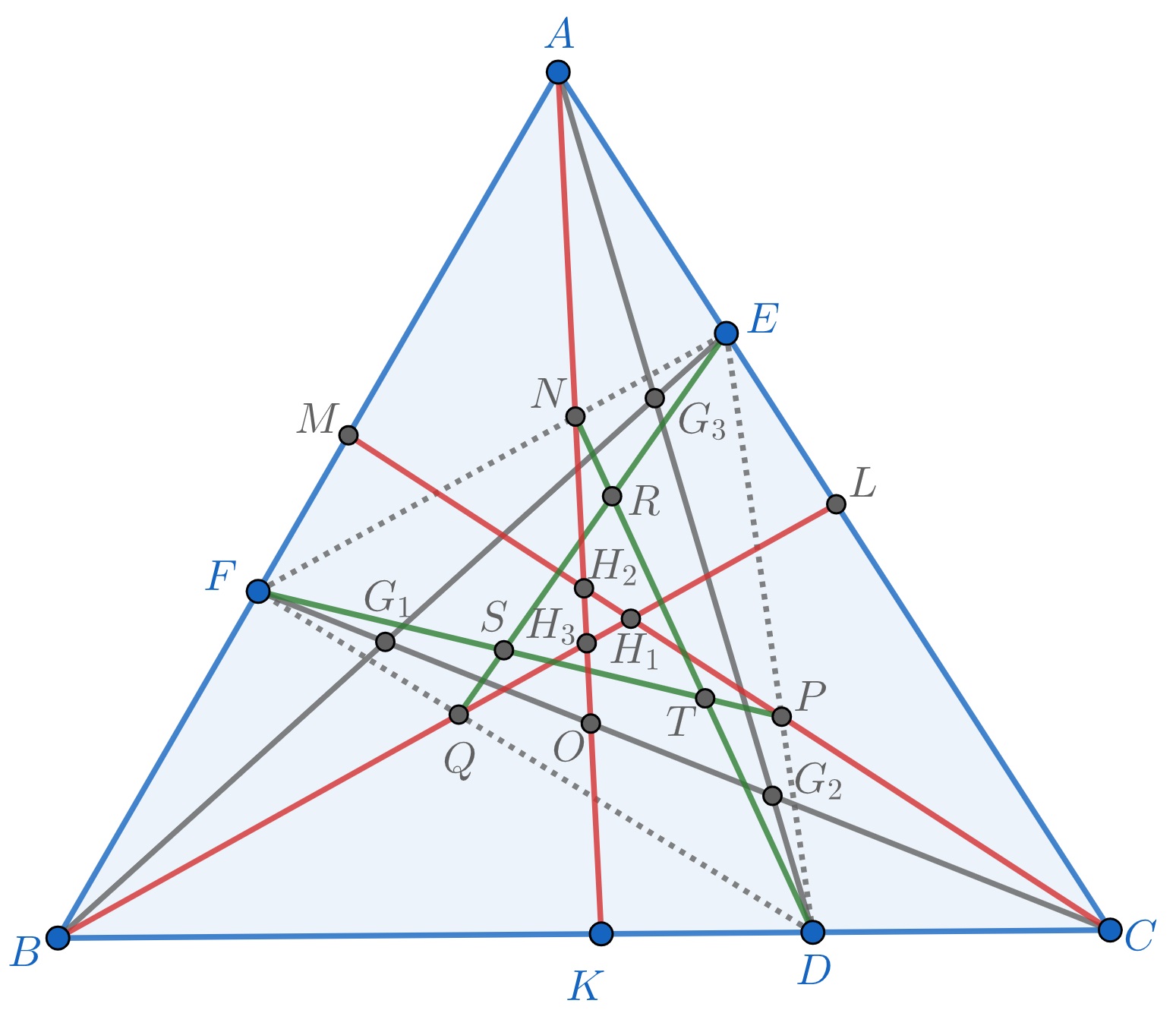}}
\label{fig3}
\caption{Inequality about the ratio of areas of $\triangle RST$ and $\triangle DEF$.}
\end{figure}
\begin{proof} Let $O$ be intersection point of lines $AK$ and $CF$ (see Fig. 3). By Menelaus' theorem $\frac{|BK|}{|KC|}\cdot\frac{|CO|}{|OF|}\cdot\frac{|FA|}{|AB|}=1$. Then $\frac{|CO|}{|OF|}=\frac{1+\lambda_3}{\lambda_3}\cdot\frac{1}{u}$. Similarly, by Menelaus' theorem $\frac{|CO|}{|OF|}\cdot\frac{|FN|}{|NE|}\cdot\frac{|EA|}{|AC|}=1$. Then $$\frac{|FN|}{|NE|}=\alpha:=\frac{u\lambda_3(1+\lambda_2)}{1+\lambda_3}.$$ Similarly, $$\frac{|DQ|}{|QF|}=\beta:=\frac{v\lambda_1(1+\lambda_3)}{1+\lambda_1},\ \frac{|EP|}{|PD|}=\gamma:=\frac{w\lambda_2(1+\lambda_1)}{1+\lambda_2}.$$ By applying formula (1) to $\triangle DEF$ and points $N,Q,P$ on its sides, and noting that orientation has changed, we obtain
$$\frac{\textnormal{Area}(\triangle RST)}{\textnormal{Area}(\triangle DEF)}=\frac{(\alpha \beta \gamma-1)^2}{(\alpha \gamma +\alpha +1)(\beta \alpha+\beta +1)(\gamma\beta +\gamma +1)}. \eqno(7)$$
By Hölder's inequality (3),
$$(\alpha \gamma +\alpha +1)(\beta \alpha+\beta +1)(\gamma \beta +\gamma +1)\ge\left(\sqrt[3]{(\alpha \beta \gamma)^2}+\sqrt[3]{\alpha \beta \gamma}+1\right)^3, \eqno(8)$$
with equality case only when $\alpha =\beta =\gamma$. From (7) and (8) it follows that
$$\frac{\textnormal{Area}(\triangle RST)}{\textnormal{Area}(\triangle DEF)}\le\frac{(\alpha \beta \gamma-1)^2}{\left(\sqrt[3]{(\alpha \beta \gamma)^2}+\sqrt[3]{\alpha \beta \gamma}+1\right)^3}. \eqno(9)$$
Since $\alpha \beta \gamma=\lambda_1 \lambda_2\lambda_3uvw$, (6) follows from (9). The equality case in (6) holds true when $$\frac{u\lambda_3(1+\lambda_2)}{1+\lambda_3}=\frac{v\lambda_1(1+\lambda_3)}{1+\lambda_1}=\frac{w\lambda_2(1+\lambda_1)}{1+\lambda_2}.$$
\end{proof}
In particular, if $\lambda_1 \lambda_2\lambda_3uvw=1$ in (6), then $\textnormal{Area}(\triangle RST)=0$ and therefore lines $DN$, $EQ$, and $FP$ are concurrent. This generalizes Schlömilch's theorem even further (Fig. 4).
\begin{corollary} Let $D$ and $K$, $E$ and $L$, $F$ and $M$ be points on sides $BC$, $AC$, and $AB$, respectively, of a triangle $ABC$. Denote $AK\cap EF=N$, $BL\cap DF=Q$, $CM\cap DE=P$. If $$\frac{|BD|}{|DC|}\cdot\frac{|CE|}{|EA|}\cdot\frac{|AF|}{|FB|}\cdot\frac{|BK|}{|KC|}\cdot\frac{|CL|}{|LA|}\cdot\frac{|AM|}{|MB|}=1,$$ then $DN$, $EQ$, and $FP$ are concurrent.
\end{corollary}
\begin{figure}[htbp]
\centerline{\includegraphics[scale=0.4]{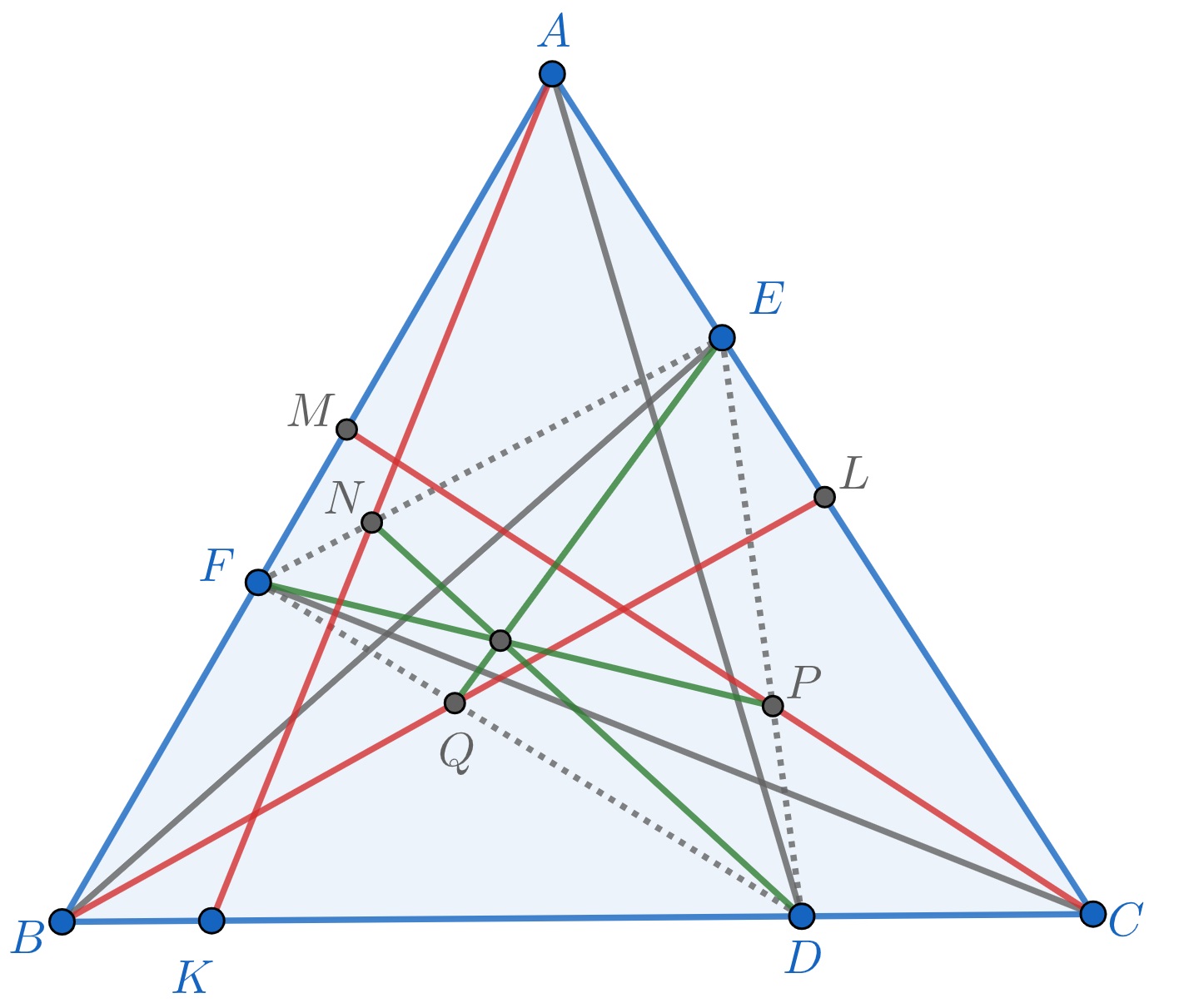}}
\label{fig35}
\caption{Generalization of Zetel's theorem.}
\end{figure}
The special case $uvw=1$ of Theorem 2.1 is also of interest (Fig. 5). 
\begin{corollary} Let $D,E,$ and $F$ be arbitrary points on sides $BC$, $AC$, and $AB$, repsectively, of a triangle $ABC$. Let cevians $AK$, $BL$, and $CM$ of triangle $ABC$ be concurrent at point $H$. Denote $AK\cap EF=N$, $BL\cap DF=Q$, $CM\cap DE=P$, $DN\cap EQ=R$, $FP\cap EQ=S$, and $FP\cap DN=T$. Denote also $\frac{|BD|}{|DC|}=\lambda_1$, $\frac{|CE|}{|EA|}=\lambda_2$, and $\frac{|AF|}{|FB|}=\lambda_3$. 
Then
$$
\frac{\textnormal{Area}(\triangle RST)}{\textnormal{Area}(\triangle DEF)}\le\frac{(\lambda_1 \lambda_2 \lambda_3-1)^2}{\left(\sqrt[3]{(\lambda_1 \lambda_2\lambda_3)^2}+\sqrt[3]{\lambda_1 \lambda_2\lambda_3}+1\right)^3}.
$$
\end{corollary}
\begin{proof}
Denote as before $\frac{|BK|}{|KC|}=u$, $\frac{|CL|}{|LA|}=v$, and $\frac{|AM|}{|MB|}=w$. 
Since $uvw=1$ (Ceva's theorem), $\alpha \beta \gamma=\lambda_1 \lambda_2\lambda_3$, and therefore the inequality follows from (6). The equality case holds true when $$u=\frac{1+\lambda_3}{1+\lambda_2}\sqrt[3]{\frac{\lambda_1\lambda_2}{\lambda_3^2}},\  v=\frac{1+\lambda_1}{1+\lambda_3}\sqrt[3]{\frac{\lambda_2\lambda_3}{\lambda_1^2}},\ w=\frac{1+\lambda_2}{1+\lambda_1}\sqrt[3]{\frac{\lambda_1\lambda_3}{\lambda_2^2}}.$$
\end{proof}
\begin{figure}[htbp]
\centerline{\includegraphics[scale=0.4]{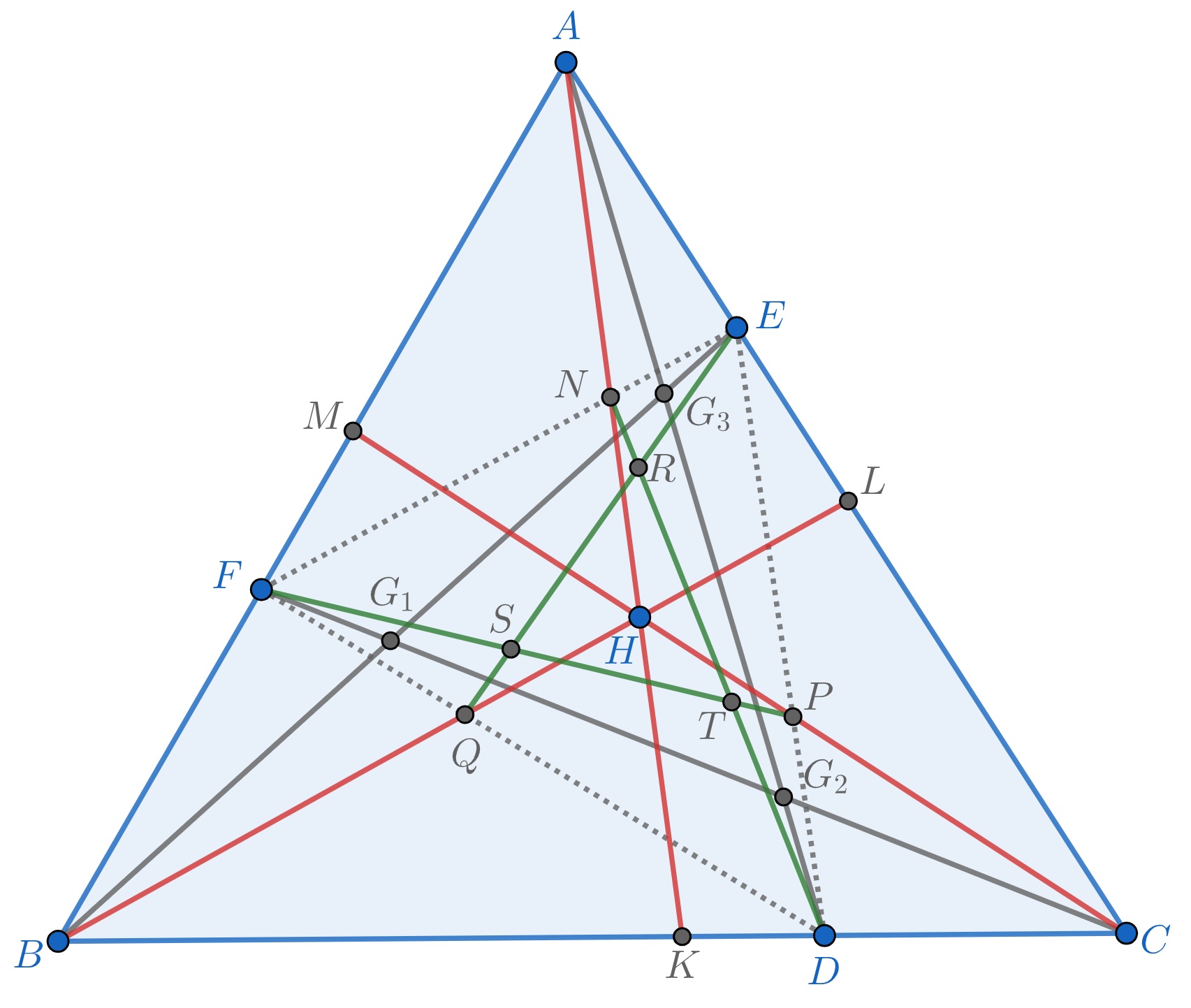}}
\label{fig5}
\caption{A new proof of Zetel's generalization of Schlömilch's theorem.}
\end{figure}
Note that if $\lambda_1 \lambda_2 \lambda_3=1$ then Corollary 2.3 implies Zetel's generalization of Schlömilch's theorem. Denote  $BE\cap CF=G_1$, $AD\cap CF=G_2$, and $AD\cap BE=G_3$. 
We can also observe that if $H\in \triangle G_1G_2G_3$, then $\triangle RST \subset \triangle G_1G_2G_3$, and therefore
$$
\textnormal{Area}(\triangle RST)<\textnormal{Area}(\triangle G_1G_2G_3). \eqno(10)
$$
In general (10) is not always true. For example, if
$\lambda_1=1$, $\lambda_2=0.001$, $\lambda_3=1$, $u=\frac{1}{3}$, $v=\frac{3}{40}$, $w=40$, then by (1), (2), and (7),
$$
\frac{\textnormal{Area}(\triangle RST)}{\textnormal{Area}(\triangle G_1G_2G_3)}=\frac{\lambda_1 \lambda_2 \lambda_3+1}{(\lambda_1 +1)(\lambda_2 +1)(\lambda_3 +1)}\times$$ $$\times\frac{(\lambda_1 \lambda_2+\lambda_1 +1)(\lambda_2 \lambda_3+\lambda_2 +1)(\lambda_3 \lambda_1+\lambda_3 +1)}{(\alpha \gamma +\alpha +1)(\beta \alpha+\beta +1)(\gamma \beta +\gamma +1)}\approx1.079>1. \eqno(11)
$$
By considering limiting cases $\lambda_1=\epsilon$, $\lambda_2={\epsilon}$, $\lambda_3=\epsilon^2$, $u=\epsilon$, $v=\epsilon$, $w=\frac{1}{\epsilon^2}$, where $\epsilon\rightarrow0^+$ and $\epsilon\rightarrow+\infty$, we can see that the ratio of areas in (11) can be arbitrarily small and arbitrarily large positive numbers, respectively.

Let us now consider special case $\lambda_1 \lambda_2 \lambda_3=1$ of the configuration in Theorem 2.1 ($G=G_1=G_2=G_3$, Fig. 6). From (7) we obtain
$$
\frac{\textnormal{Area}(\triangle RST)}{\textnormal{Area}(\triangle DEF)}=\frac{(uvw-1)^2}{(\alpha \gamma +\alpha +1)(\beta \alpha+\beta +1)(\gamma \beta +\gamma +1)}. \eqno(12)
$$
Denote $BL\cap CM=H_1$, $AK\cap CM=H_2$, and $AK\cap BL=H_3$. By equality (1) for $\triangle H_1H_2H_3$,
$$\frac{\textnormal{Area}(\triangle ABC)}{\textnormal{Area}(\triangle H_1H_2H_3)}=\frac{(uv+u +1)(vw+v+1)(wu+w+1)}{(uvw-1)^2}. \eqno(13)$$ By multiplying equalities (2), (12), and (13), we obtain
$$
\frac{\textnormal{Area}(\triangle RST)}{\textnormal{Area}(\triangle H_1H_2H_3)}=\frac{(uv+u +1)(vw+v+1)(wu+w+1)}{(\alpha \gamma +\alpha +1)(\beta \alpha+\beta +1)(\gamma \beta +\gamma +1)}\times
$$
$$\times\frac{2}{(\lambda_1 +1)(\lambda_2 +1)(\lambda_3 +1)}. \eqno(14)$$
We observe that if $G\in \triangle H_1H_2H_3$, then $\triangle RST \subset \triangle H_1H_2H_3$, and therefore
$$
\textnormal{Area}(\triangle RST)<\textnormal{Area}(\triangle H_1H_2H_3). \eqno(15)
$$
In general (15) is not always true. For example, if
$\lambda_1=1$, $\lambda_2=1$, $\lambda_3=1$, $u=0.01$, $v=1$, $w=20$, then by (14),
$$
\frac{\textnormal{Area}(\triangle RST)}{\textnormal{Area}(\triangle H_1H_2H_3)}\approx 1.19>1. \eqno(16)
$$
\begin{figure}[htbp]
\centerline{\includegraphics[scale=0.4]{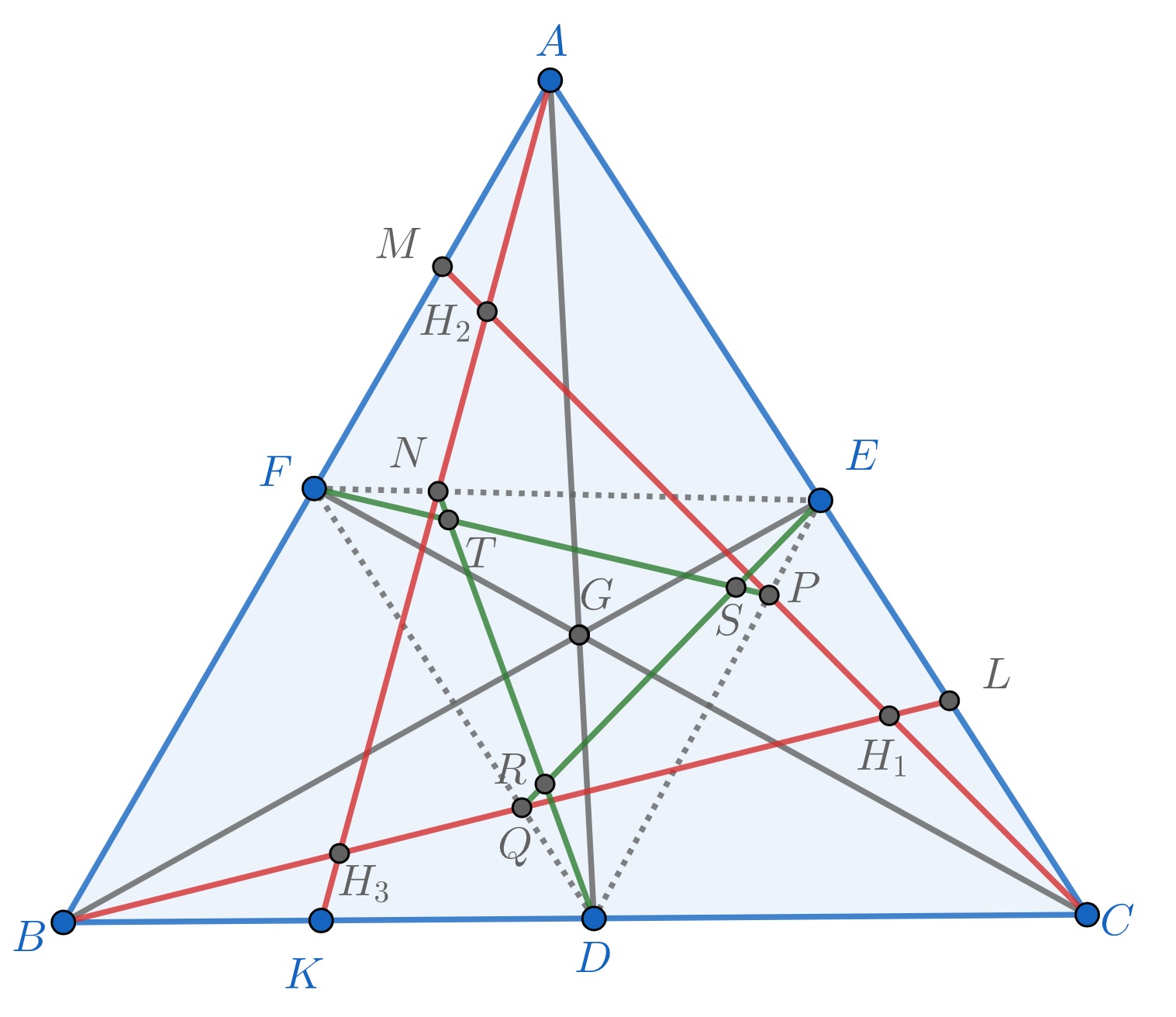}}
\label{fig6}
\caption{Comparison of areas of $\triangle RST$ and $\triangle H_1H_2H_3$.}
\end{figure}
By considering limiting cases $\lambda_1=\lambda_2=\lambda_3=1$, $u=\epsilon$, $v=1$, $w=\frac{1}{\epsilon^2}$, where $\epsilon\rightarrow0^+$ and $\epsilon\rightarrow+\infty$, we can see that the ratio of areas in (16) can be arbitrarily large and arbitrarily small positive numbers, respectively.

We will now return to configuration in Fig. 1. A.F. Möbius considered areas in the special case where $AD$, $BE$, and $CF$ are concurrent \cite{mob} p. 198. J.F. Rigby's inequality (4) generalized this result \cite{rigby} (see also \cite{mit} p. 340). The following theorem is further generalization of these two results.
\begin{theorem} Let $D,E,$ and $F$ be arbitrary points on sides $BC$, $AC$, and $AB$, repsectively, of a triangle $ABC$. Denote  $BE\cap CF=G_1$, $AD\cap CF=G_2$, and $AD\cap BE=G_3$.
Let areas of $\triangle AEF$, $\triangle BFD$, $\triangle CDE$, $\triangle DEF$, and $\triangle G_1G_2G_3$ be $p$, $q$, $r$, $x$, and $y$ (Figure 1). Then
$$
x^3+(p+q+r)x^2-4pqr\ge x^2y.
$$
\end{theorem}
\begin{proof} Denote $\frac{|BD|}{|DC|}=\lambda_1$, $\frac{|CE|}{|EA|}=\lambda_2$, and $\frac{|AF|}{|FB|}=\lambda_3$. 
Then the left side of the inequality can be written as (see \cite{mit} p. 340) 
$$
x^3+(p+q+r)x^2-4pqr=\frac{(\lambda_1 \lambda_2 \lambda_3-1)^2}{(\lambda_1 +1)^2(\lambda_2 +1)^2(\lambda_3 +1)^2}\cdot \textnormal{Area}(\triangle ABC).
$$
By (1) and (2), this can also be written as
$$
x^3+(p+q+r)x^2-4pqr=x^2y\cdot \frac{(\lambda_1 \lambda_2+\lambda_1 +1)(\lambda_2 \lambda_3+\lambda_2 +1)(\lambda_3 \lambda_1+\lambda_3 +1)}{(\lambda_1 \lambda_2 \lambda_3+1)^2}.
$$
The values of the last fraction change in interval $(1,+\infty)$ (consider limiting case $\lambda_1, \lambda_2\rightarrow 0$, $0<\lambda_3<+\infty$), and therefore inequality (5) holds true.
\end{proof}
From this proof it follows that $\lambda =1$ is the best constant for inequality
$$
x^3+(p+q+r)x^2-4pqr\ge \lambda x^2y.
$$
\textbf{Open problem.} In the case $uvw=1$ (Corollary 2.3, Figure 5), prove that the vertices of the triangle formed by lines $G_1S$, $G_2T$, and $G_3R$ are on lines $AK$, $BL$, $CM$. 

\section{Conclusion}
In the paper the use of geometric area inequalities for the proof of theorems about concurrency of lines is demonstrated. O. Schlömilch's theorem about concurrent cevians of a triangle and its generalization by S.I. Zetel were generalized further using the area method. The proved general inequality is also explored in special cases of concurrent cevians. Also a refinement of J.F. Rigby's inequality is proved.
\section*{Acknowledgments}
This work was supported by ADA University Faculty Research and Development Fund. I thank Francisco Javier García Capitán for the discussion of the configuration and informing me about the possibility of the use barycentric coordinates.

\section{Declarations}
\textbf{Ethical Approval.}
Not applicable.
 \newline \textbf{Competing interests.}
None.
  \newline \textbf{Authors' contributions.} 
Not applicable.
  \newline \textbf{Funding.}
This work was completed with the support of ADA University Faculty Research and Development Fund.
  \newline \textbf{Availability of data and materials.}
Not applicable

\end{document}